\documentclass{article}
\usepackage{xcolor}
\definecolor{grey}{rgb}{0.5,0.5,0.5}
\usepackage{amsmath}
\usepackage[left=2.6cm, right=2.6cm, bottom=2.6cm, top=2.6cm]{geometry}
\usepackage{hyperref}
\usepackage{cleveref}
\usepackage[dvipsnames]{xcolor}
\usepackage{cancel,xcolor}
\makeatletter
\newcommand{\cancelslash}[2]{%
  \tikz[baseline=(X.base)]{
    \node[inner sep=0pt] (X) {$#2$};
    \draw[line width=1pt, color=#1] (X.south west) -- (X.north east);
  }%
}

\usepackage{dsfont}
\usepackage{amssymb}
\usepackage{mathtools}
\usepackage{tikz-cd}
\usepackage{amsthm}
\usepackage{graphicx}
\usepackage{graphics}
\usepackage{mathrsfs}
\usepackage{cleveref}
\usepackage{thmtools}
\usepackage{enumitem}
\usepackage{bbm}
\usepackage{dsfont}
\usepackage[toc,page]{appendix}

\newlist{steps}{enumerate}{1}
\usepackage{mathtools}
\makeatletter
\newcommand{\xMapsto}[2][]{\ext@arrow 0599{\Mapstofill@}{#1}{#2}}
\def\Mapstofill@{\arrowfill@{\Mapstochar\Relbar}\Relbar\Rightarrow}
\makeatother
\usepackage{array}
\setlist[steps, 1]{label = Step \arabic*:}
\theoremstyle{plain}
\usepackage{footnote}
\makesavenoteenv{tabular}
\newtheorem*{theorem*}{Theorem}
\newtheorem{theorem}{Theorem}[subsection]
\newtheorem{definition}[theorem]{Definition}

\newtheorem{proposition}[theorem]{Proposition}
\newtheorem{remark}[theorem]{Remark}
\newtheorem{corollary}[theorem]{Corollary}
\newtheorem*{corollary*}{Corollary}
\newtheorem*{proposition*}{Proposition}
\newtheorem{definition*}{Definition}
\newtheorem{exmp}[theorem]{Example}

\newtheorem{lemma}[theorem]{Lemma}
\newtheorem*{lemma*}{Lemma}

\numberwithin{equation}{subsection}
\usepackage{rotating}
\usepackage{comment}
\usepackage[dvipsnames]{xcolor}
\definecolor{darkred}{RGB}{139,0,0}
\definecolor{darkgreen}{RGB}{0,100,0}
\definecolor{darkblue}{RGB}{0,0,139}
\definecolor{Slash1}{RGB}{230,25,75}    
\definecolor{Slash2}{RGB}{60,180,75}    
\definecolor{Slash3}{RGB}{0,130,200}    
\definecolor{Slash4}{RGB}{130,200,200}   
\definecolor{Slash5}{RGB}{0,0,250}   
\definecolor{Slash6}{RGB}{70,240,240}   
\definecolor{Slash7}{RGB}{155,125,25}   
\definecolor{Slash8}{RGB}{210,245,60}   
\definecolor{Slash9}{RGB}{240,180,250}   
\definecolor{Slash10}{RGB}{255,180,195} 
\usepackage{mathabx,epsfig}
\def\acts{\mathrel{\reflectbox{$\righttoleftarrow$}}}
\usepackage[nottoc,notlot,notlof]{tocbibind} 

\usepackage[dvipsnames]{xcolor}
\usepackage{cancel}

\definecolor{darkred}{RGB}{40,0,0}
\definecolor{darkgreen}{RGB}{0,80,0}
\definecolor{darkblue}{RGB}{0,0,110}

\title{Isometric Embeddings and Hyperk\"{a}hler Geometry of $\textup{T}^*\mathbb{CP}^{n-1}$ via the Scheme of Rank-1 Projections}
\author{Joshua Lackman}
\date{}
\begin{document}

\maketitle
\begin{abstract}
\noindent 
We show that the hyperk\"{a}hler geometry of $\textup{T}^*\mathbb{CP}^{n-1}$ can be described algebraically by the affine scheme of rank–1 projections, and that this description simultaneously yields explicit $SU(n)$–equivariant isometric embeddings
\[
\textup{T}^*\mathbb{CP}^{n-1}\xhookrightarrow{}\mathbb{R}^{(n^2+1)^2}\,,
\]
as well as a generalization of the hyperk\"{a}hler geometry of $\textup{T}^*\mathbb{CP}^{n-1}$ to arbitrary commutative rings with involutions (and some noncommutative ones). In particular, we obtain para–hyperk\"{a}hler and complex hyperk\"{a}hler manifolds by taking the rings to be the split complex numbers and bicomplex numbers, respectively. The functor of points of the scheme of rank–1 projections is the functor that maps a commutative ring $\mathcal{R}$ to the space of idempotents in $M_n(\mathcal{R})$ whose images are rank–1 projective modules. In particular, its space of $\mathbb{C}$–points is identified with $\textup{T}^*\mathbb{CP}^{n-1}.$
\end{abstract}
\tableofcontents
\section{Introduction}
We show that the hyperk\"{a}hler geometry of $\textup{T}^*\mathbb{CP}^{n-1}$ (\ref{cal}, \ref{egu}) can be described algebraically by the affine scheme of rank–1 projections,\footnote{The earliest reference that the author could find for some identification $\textup{T}\mathbb{CP}^{n-1}\cong\{\textup{projections on } \mathbb{C}^n\}$ is from 2021 (\ref{leu}).} and that this algebraic description simultaneously yields:
\begin{enumerate}
    \item explicit $SU(n)\times\mathbb{Z}/2\mathbb{Z}$–equivariant isometric embeddings
$\textup{T}^*\mathbb{CP}^{n-1}\xhookrightarrow{}\mathbb{R}^{(n^2+1)^2}\,,$\footnote{The only other Euclidean embedding of $\textup{T}^*\mathbb{CP}^{n}$ in the literature that the author could find is for $n=1$ (\ref{han}). In the pseudo–Riemannian context, one can be found in (\ref{dun}), also for $n=1.$}
\item a generalization of the hyperk\"{a}hler geometry of $\textup{T}^*\mathbb{CP}^{n-1},$ where $\mathbb{C}$ is replaced by an arbitrary commutative ring with involution (some noncommutative rings work too).
\end{enumerate}
Regarding item 2, for a commutative ring with involution $(\mathcal{R},*),$ the quaternions are replaced by 
\begin{equation}
    \mathcal{R}[x,*]/(x^2+1)\,,
\end{equation}
where $\mathcal{R}[x,*]$ is the skew–polynomial ring, ie. $xr=r^*x$ for all $r\in\mathcal{R}.$ For example, let $\mathcal{R}$ be the algebraic complex numbers $\overline{Q}.$ Then each Zariski tangent space is endowed with the structure of a module with respect to the quaternion algebra over $\overline{Q}\cap\mathbb{R},$ and these modules are compatible with a metric valued in $\overline{Q}\cap \mathbb{R}.$ 
\\\\In particular, specializing $(\mathcal{R},*)$ to the bicomplex numbers, ie. 
\begin{equation}
   \mathbb{C}[x]/(x^2+1)\,,\;\; (a+xb)^*=\bar{a}+x\bar{b}\,, 
\end{equation}
we obtain a complex hyperk\"{a}hler manifold that complexifies $\textup{T}^*\mathbb{CP}^{n-1}.$ A complex hyperk\"{a}hler manifold is analogous to a hyperk\"{a}hler manifold, with the quaternions replaced by its complexification, ie. the biquaternions. These come with a pseudo–Riemannian metric and every point in 
\begin{equation}
    \{(x,y,z)\in\mathbb{C}^3:x^2+y^2+z^2=1\}
\end{equation}
determines a compatible integrable almost complex structure. If instead we use the split complex numbers, ie.
\begin{equation}
    \mathbb{R}[x]/(x^2-1)\,,\;\;(a+xb)^*=a-xb\,,
\end{equation}
we obtain a para–hyperk\"{a}her manifold.
\\\\Para–hyperk\"{a}hler manifolds and complex hyperk\"{a}hler manifolds carry many symplectic forms together with transverse pairs of real and complex polarizations, making them natural candidates for quantization. Indeed, a transverse pair of polarizations determines a symplectic connection and hence, via Fedosov quantization (\ref{fed}), a formal deformation quantization.
\\\\Complex hyperk\"{a}hler manifolds are particularly interesting due to an analogue of Berger's theorem: the main result of \ref{bou} states that the semisimple part of the algebra of parallel endomorphisms of an indecomposable pseudo–Riemannian manifold is a Clifford algebra, and that the maximal Clifford algebra it can be is the biquaternions, which exactly corresponds to a complex hyperk\"{a}hler manifold. However, there is a lack of explicit examples in the literature. Some discussion of para–hyperk\"{a}hler geometry and complex hyperk\"{a}hler geometry can be found in \ref{bou}, \ref{maz}.
\subsection{Results}
In particular: 
\begin{enumerate}
    \item We construct explicit $SU(n)\times\mathbb{Z}/2\mathbb{Z}$–equivariant isometric embeddings 
\begin{equation}\label{isome}
   \textup{T}^*\mathbb{C}\mathbb{P}^{n-1}\xhookrightarrow{}\mathbb{R}^{(n^2+1)^2}
\end{equation}
with respect to Calabi's hyperk\"{a}hler metric\footnote{The action $SU(n)\acts \textup{T}^*\mathbb{CP}^{n-1}$ is by isometries and is induced by the action $SU(n)\acts \mathbb{CP}^{n-1}.$ There is an additional $\mathbb{Z}/2\mathbb{Z}$ isometry.} on $ \textup{T}^*\mathbb{C}\mathbb{P}^{n-1},$ for all $n>1.$ The images of these maps are connected components of real affine varieties. 
\\\\To understand this embedding:
\begin{enumerate}[label=\arabic*.]
    \item Let $\textup{Herm}(n^2)$ denote $n^2\!\times\! n^2$ Hermitian matrices acting on the Hilbert space of $n\!\times\! n$ matrices, $\textup{M}_n(\mathbb{C}),$ with respect to the standard basis.
    \item Let $g_0,g_1,g_2$ denote the standard inner products on $\mathbb{R},\,\textup{M}_n(\mathbb{C}),\,\textup{Herm}(n^2).$
    \item We identify a point in $\textup{T}^*\mathbb{CP}^{n-1}$ with a pair $(V,f),$ where $V\subset \mathbb{C}^n$ is a 1–dimensional subspace and $f\in\textup{hom}(V^\perp,V).$
\end{enumerate}
Consider the following Euclidean space:
\begin{equation}
\big(\mathbb{R}\oplus \textup{M}_n(\mathbb{C})\oplus\textup{Herm}(n^2),\,\frac{5}{2}g_0\oplus g_1\oplus \frac{1}{2}g_2\big)\,.
\end{equation}
Then $\textup{T}^*\mathbb{CP}^{n-1}$ is isometric to the submanifold given by all $(x,q,p)$ satisfying
\begin{equation}\label{semi}
    p(q)=xq,\,p^2=xp,\,xq^2=q,\,x\textup{Tr}\,q=1,\,\textup{Tr}\, p=x,\,\textup{Tr}(q^\dagger q)=x^2,\,x\ge 0\,,
     \end{equation}
and a K\"{a}hler potential is given by
\begin{equation}
    (x,q,p)\mapsto x^2 \,.
\end{equation}
In other words: $xq$ is a rank–1 projection, $p/x$ equals the \textit{Hermitian} projection onto $\mathbb{C}q$ and $x$ is the trace–class norm of $q.$ The map 
\begin{equation}
\mathbb{R}\oplus\textup{M}_n(\mathbb{C})\oplus \textup{Herm}(n^2)\to \textup{M}_n(\mathbb{C})\,,\;\;   (x,q,p)\mapsto xq
\end{equation} restricts to a diffeomorphism from the submanifold \ref{semi} onto rank–1 projections. The embedding \ref{isome} is then obtained by identifying $\textup{T}^*\mathbb{CP}^{n-1}$ with rank–1 projections, which we do as follows: 
\begin{equation}
  \textup{Rank}–1\,\textup {Projections}\to  \textup{T}^*\mathbb{CP}^{n-1}\,,\;\;q\mapsto (\textup{im}(q),\,q\vert_{\textup{im}(q)^\perp})\,.
\end{equation} 
A quantum–mechanically well–motivated K\"{a}hler potential is given by the trace–class norm of the projection, ie. the function
\begin{equation}
    (V,f)\mapsto \sqrt{1+ff^*}\,.
\end{equation}
\item We show that the following space is naturally a complex hyperk\"{a}hler manifold:
\begin{equation}\label{spce}
    \{(q,z)\in\textup{M}_n(\mathbb{C}[x]/(x^2+1))\times\mathbb{R}[x]/(x^2+1):q^2=q,\,\textup{Tr}(q)=1,\,z^2\,\textup{Tr}(q^\dagger q )=1\}\,.
\end{equation}
\end{enumerate}
To achieve this, we study the hyperk\"{a}hler geometry of $\textup{T}^*\mathbb{CP}^{n-1}$ in the context of $^*$–algebras: we show that $\textup{T}^*\mathbb{CP}^{n-1}$ is biholomorphic to the $\mathbb{C}$–points of the affine scheme of rank–1 projections and that any involution $*$ on a commutative ring $\mathcal{R}$ determines an analogue of the hyperk\"{a}hler structure of $\textup{T}^*\mathbb{CP}^{n-1}$ on the $\mathcal{R}$–points of this scheme.
We obtain \ref{spce} by letting $\mathcal{R}$ be the bicomplex numbers, ie. $\mathbb{C}[x]/(x^2+1)$ with involution $(a+xb)^*=\overline{a}+x\overline{b}.$
\newpage\subsection{Overview}
The main idea is the following: the functor of points of the scheme of rank–1 projections is given by
\begin{equation}
\mathds{1}^{n-1}:\textup{CRing}\to\textup{Set}\,,\;\;\mathcal{R}\mathds{1}^{n-1}:=\{q\in\textup{M}_n(\mathcal{R}):q^2=q\,\textup{ and }\,\textup{im}(q)\,\textup{ is a rank–1 projective module}\}\,.
\end{equation}
This functor is representable by
\begin{equation}
    \mathbb{Z}[\{x_j^i\}_{i,j=1}^n]/I\,,
\end{equation}
where $I$ is the ideal generated by 
\begin{equation}
\sum_{i=1}^n x^j_ix^i_k-x^j_k\,,\;\;\sum_{i=1}^n x^i_i-1\,,\;\; x^j_kx^l_m-x^j_lx^k_m\,,\;\;\textup{for all}\;1\le j,k,l,m\le n\,.
\end{equation}
The tangent bundle of $\mathds{1}^{n-1}$ is the scheme 
\begin{equation}
\textup{T}\mathds{1}^{n-1}:\textup{CRing}\to\textup{Set}\,,\;\;\textup{T}\mathcal{R}\mathds{1}^{n-1}:=\mathcal{R}[\varepsilon]/(\varepsilon^2)\mathds{1}^{n-1}\,.
\end{equation}
There is a symplectic form on this scheme, which for $\mathcal{R}\in\textup{CRing}$ is given by\footnote{$ \textup{T}_q\mathcal{R}\mathds{1}^{n-1}=\{a\in\textup{M}_n(\mathcal{R}):qa+aq=a\}.$ }
\begin{equation}
    \Omega_q:\textup{T}_q\mathcal{R}\mathds{1}^{n-1}\otimes \textup{T}_q\mathcal{R}\mathds{1}^{n-1}\to \mathcal{R}\,,\;\;\Omega_q(a,b)\mapsto \textup{Tr}(q[a,b])\,,
\end{equation}
and letting $*:\mathcal{R}\to\mathcal{R}$ be any involution, there is a locally defined Hermitian metric given by
 \begin{equation}
     h_q:\textup{T}_q\mathcal{R}\mathds{1}^{n-1}\otimes \textup{T}_q\mathcal{R}\mathds{1}^{n-1}\to \mathcal{R}\,,\;\;h_q(a,b)=\frac{2}{\textup{Tr}(q^\dagger q)^{1/2}}\textup{Tr}(ab^\dagger)-\frac{1}{\textup{Tr}(q^\dagger q)^{3/2}}\textup{Tr}(aq^\dagger )\textup{Tr}(qb^\dagger)\,.
     \end{equation}
These are related by $h(a,b)=\Omega(a,I(b)),$ where
\begin{equation}
    I_q:\textup{T}_q\mathcal{R}\mathds{1}^{n-1}\to \textup{T}_q\mathcal{R}\mathds{1}^{n-1}\,,\;\;I_q(a):=\frac{[a^{\dagger},q]}{\textup{Tr}(q^{\dagger} q)^{1/2}}+\frac{\textup{Tr}(a^\dagger q)}{2}\frac{[q,q^\dagger]}{\textup{Tr}(q^{\dagger} q)^{3/2}}
\end{equation}
is integrable, satisfies $I^2=-1$ and $Ir=r^*I$ for all $r\in\mathcal{R}.$ Therefore, $\textup{T}_q\mathcal{R}\mathds{1}^{n-1}$ is a module over $\mathcal{R}[x,*]/(x^2+1),$ where $\mathcal{R}[x,*]$ is the skew–polynomial ring, ie. $xr=r^*x$ for all $r\in\mathcal{R}.$ For $\mathcal{R}=\mathbb{C},$ this is the quaternions and we obtain the hyperk\"{a}hler structure of $\textup{T}^*\mathbb{CP}^{n-1}.$
\\\\These formulas are only locally defined because they depend on a square root of $\textup{Tr}(q^\dagger q),$ for which there isn't a canonical choice unless $\textup{Tr}(q^\dagger q)>0.$ However, these formulas are globally defined on the ``cover"
\begin{equation}
    \textup{CRing}_*\to \textup{Set}\,,\;\;\mathcal{R}\tilde{\mathds{1}}^{n-1}:=\{(q,r)\in \mathcal{R}\mathds{1}^{n-1}\times\mathcal{R}:r=r^*\,\textup{ and }\,r^2\textup{Tr}(q^\dagger q)=1\}\,,
\end{equation}
with a potential given by $(q,r)\mapsto r\textup{Tr}(q^\dagger q).$\footnote{$\textup{CRing}_*$ is the category of commutative rings with involutions.} These formulas make sense for the quaternions as well, with $h,\Omega$ replaced by their real parts.
\\\\Specializing to $\mathcal{R}=\mathbb{C},$ the Hermitian metric \ref{h0} isometrically embeds into $\textup{M}_n(\mathbb{C})-\{0\},$ where the latter is equipped with the Hermitian metric defined by the same formula. Its real part is of the following form, with respect to the Hilbert–Schmidt inner product: let $\mathcal{H}$ be a Hilbert space. We have a Riemannian metric $g$ on $\mathcal{H}-\{0\},$ given by 
\begin{equation}\label{h0}
 g_x(v,w):=\textup{Re}\Big(\frac{2}{\|x\|}\langle v,w\rangle-\frac{1}{\|x\|^3}\langle v,x\rangle\langle x,w\rangle\Big)\,.
\end{equation}
Therefore, to get our desired isometric embedding into Euclidean space, it is enough to isometrically embed $g$ into Euclidean space for an arbitary Hilbert space, which we do as follows:
let $g_\mathbb{R},g_{\mathcal{H}},g_{\mathcal{L}(\mathcal{H})}$ be the standard real inner products on $\mathbb{R},\mathcal{H},\mathcal{L}(\mathcal{H}),$ respectively. Let $\mathbb{R}\oplus\mathcal{H}\oplus\mathcal{L}(\mathcal{H})_{\textup{self–adjoint}}$ be equipped with the Riemannian metric given by 
\begin{equation}
    \frac{5}{2}g_{\mathbb{R}}\oplus g_{\mathcal{H}}\oplus\frac{1}{2}g_{\mathcal{L}(\mathcal{H})}\,.
\end{equation}
Then  
\begin{equation}
(\mathcal{H}-\{0\},g)\to \mathbb{R}\oplus\mathcal{H}\oplus\mathcal{L}(\mathcal{H})_{\textup{self–adjoint}}\,,\;\;x\mapsto \|x\|^{1/2}(1,\|x\|^{-1}x,\|x\|^{-2}x\otimes x^*)
\end{equation}
is an isometric embedding, where $(x\otimes x^*)(y):=\langle x,y\rangle x.$
\subsection{Main Operator}
We define the main operator used in this paper, from which we will obtain the almost complex structure. For the most part, we work at the level of $^*$–rings, for which matrix $^*$–rings are a special case. For a ring $\mathcal{A},$ we think of the idempotents in $\mathcal{A}[\varepsilon]/(\varepsilon^2)$ as the tangent bundle of the idempotents in $\mathcal{A}$:
\begin{definition}
Let $(\mathcal{A},*)$ be a $^*$–ring.\footnote{That is, $^*$ is an additive involution and for all $x,y\in\mathcal{A},\,(xy)^*=y^*x^*.$} For all idempotents $q+\varepsilon a \in \mathcal{A}[\varepsilon]/(\varepsilon^2)$ such that $qq^*q\in Z(\mathcal{A})q,$ let
\begin{equation}
\mathcal{L}_q(a):=2[a^*,qq^*q]+[qa^*q,q^*]\,.
\end{equation}   
\end{definition}
The crux of this paper is the following identity, \cref{themainl}:
\begin{lemma}\label{mainnn}
In the context of the previous definition: $q+\varepsilon \mathcal{L}_q(a)$ is an idempotent and writing $qq^*q=rq$ for $r\in Z(\mathcal{A}),$ we have that
\begin{equation}\label{themainl}
    \mathcal{L}_q^2(a)=-4r^3a\,.
    \end{equation}
Furthermore, for all $s\in Z(\mathcal{A}),$
\begin{equation}\label{anticomp}
 \mathcal{L}_q(sa)=s^*\mathcal{L}_q(a)\,.  
\end{equation}
\end{lemma}
\begin{proof}
We prove this result in \cref{mainproof}.
\end{proof}
For example, let $(\mathcal{R},*)$ be either a commutative ring with involution or a quaternion algebra.  Then the hypothesis of the previous definition is satisfied for all idempotents in $(M_n(\mathcal{R}),\,\dagger)$ whose images are rank–1 as right $\mathcal{R}$–modules. 
\section{Star Rings}
We will discuss the necessary theory of $^*$–rings. In particular, $^*$–rings come with a notion of rank–1 projection, which have a hyperk\"{a}hler–like structure.\footnote{Many of the definitions we give, eg. that of states, are just generlizations of the analogous $C^*$–algebra notion. For a textbook treatment, see \ref{mur}.} 
\subsection{Basic Definitions}
\begin{definition}
A $^*$–ring is a pair $(\mathcal{A},*),$ where
\begin{enumerate}
    \item $\mathcal{A}$ is a ring,
    \item $*:\mathcal{A}\to\mathcal{A}$ is an additive involution such that $(ab)^*=b^*a^*$ for all $a,b\in\mathcal{A}.$
\end{enumerate}
A morphism of $^*$–rings $f:(\mathcal{A},*_{\mathcal{A}})\to(\mathcal{B},*_{\mathcal{B}})$is a morphism of the underlying rings such that 
\begin{equation}\label{center}
f(Z(\mathcal{A}))\subset Z(\mathcal{B})
\end{equation}
and such that for all $x\in\mathcal{A},\,f(x^{*_{\mathcal{A}}})=f(x)^{*_{\mathcal{B}}}.$
\\\\We call $x\in\mathcal{A}$ self–adjoint if $x^*=x.$\footnote{The condition \ref{center} is not standard, but we will be needing it.}
\end{definition}
\begin{definition}
We let $\textup{CAlg}_{(\mathcal{R},*)}$ denote the category of commutative algebras with involutions over $(\mathcal{R},*).$ That is, the objects are pairs $(\mathcal{S},*_{\mathcal{S}})$ such that 
\begin{enumerate}
    \item $\mathcal{S}$ is a commutative $\mathcal{R}$–algebra,
    \item $*_{\mathcal{S}}$ is an involution and homomorphism on $\mathcal{S}$,
    \item the canonical morphism $f:\mathcal{R}\to\mathcal{S}$ satisfies $f(r^*)=f(r)^{*_\mathcal{S}}.$
\end{enumerate}  Morphisms are morphisms of the underlying $\mathcal{R}$–algebras that respect the involutions. We let $\textup{Set}_{*}$ denote the category of sets with involutions. Morphisms are morphisms of the underlying sets that respect the involutions.
\end{definition}
\subsubsection{Examples of $^*$–Rings}
\begin{exmp}\label{matrixex}
Let $(\mathcal{A},*)$ be a $^*$–ring, eg. $\mathbb{R},\mathbb{C},\mathbb{H}.$ Then $\textup{M}_n(\mathcal{A})$ is another $^*$–ring with the  involution given by the conjugate–transpose $\dagger.$
\end{exmp}
This next example will be very important:
\begin{exmp}
Let $(\mathcal{A},*)$ be a $^*$–ring. Then $\mathcal{A}[x]/(x^2+r)$ is another $^*$–ring with involutions given by $(a+xb)^*=a^*\pm xb^*.$
\end{exmp}
More generally, the tensor product of $^*$–rings with the same centers is a $^*$–ring:
\begin{exmp}
Let $(\mathcal{A},*_{\mathcal{A}}),\,(\mathcal{B},*_{\mathcal{B}})$ be $^*$–rings with an isomorphism $(Z(\mathcal{A}),*_{\mathcal{A}})\cong (Z(\mathcal{B}),*_{\mathcal{B}}).$ Then $\mathcal{A}\otimes_{Z(\mathcal{A})}\mathcal{B}$ is a $^*$–ring with the involution given by
$(a\otimes b)^*=a^{*_{\mathcal{A}}}\otimes b^{*_{\mathcal{B}}}.$
\end{exmp}
Quaternion algebras are of the following form:
\begin{exmp}
Let $(\mathcal{R},*)$ be a commutative $^*$–ring with involution and consider the skew–polynomial ring $(\mathcal{R}[x],*),$ ie. $xr=r^*x$ for all $r\in\mathcal{R}.$ Let $r\in\mathcal{R}$ be self–adjoint. Then $(\mathcal{R}[x],*)/(x^2+r)$ is a $^*$–ring with the involution given by
\begin{equation}
    (a+xb)^*=a^*-xb\,.
\end{equation}
We have a ``norm" given by
\begin{equation}
N:(\mathcal{R}[x],*)/(x^2+r)\to \mathcal{R}_{\textup{self–adjoint}}\,,\;\;N(w)=w^*w\,,
\end{equation}
ie. $N(wz)=N(w)N(z)$ for all $w,z\in \mathcal{R}[x]/(x^2+r).$
\end{exmp}
The examples of $^*$–rings that we will focus on are of the form $\textup{M}_n(\mathcal{R}),\,\textup{M}_n((\mathcal{R}[x],*)/(x^2+r)),$ for a commutative ring $\mathcal{R}$ with involution.
\subsection{Projection Space}
We will now define projection space.
\begin{definition}
Let $(\mathcal{A},*)$ be a $^*$–ring. We define the set of projections to be
\begin{equation}
    \mathbb{I}(\mathcal{A}):=\{q\in\mathcal{A}:q^2=q\}.
\end{equation}
We define the set of rank–1 projections to be
\begin{equation}
    \mathds{1}(\mathcal{A}):=\{q\in \mathbb{I}(\mathcal{A}):\textup{The map }Z(\mathcal{A})\to q\mathcal{A}q,\,r\mapsto rq,\textup{ is a bijection}\}\,.
\end{equation}
We let $\mathbb{P}(\mathcal{A})$ denote the set of self–adjoint rank–1 projections.
\end{definition}
For any $q\in\mathbb{I}(\mathcal{A}),$ $q\mathcal{A}q$ is an algebra over $Z(\mathcal{A}),$ called the corner algebra. We can alternatively say $q$ is a rank–1 projection if
\begin{equation}
    Z(\mathcal{A})\cong q\mathcal{A}q\,.
\end{equation}
Rank–1 projections satisfy the hypothesis of \cref{mainnn} since, in particular, $qq^*q\in Z(\mathcal{A})q.$
\\\\For a given $^*$–ring $(\mathcal{A},*),$ we have functors $\textup{CAlg}_{(Z(\mathcal{A}),*)}\to\textup{Set}_*$ given by
\begin{align}
&(\mathcal{R},*)\mapsto \mathbb{I}(\mathcal{A}\otimes_{Z(\mathcal{A})}\mathcal{R}) \,, 
\\&(\mathcal{R},*)\mapsto \mathds{1}(\mathcal{A}\otimes_{Z(\mathcal{A})}\mathcal{R})    \,.
\end{align} 
\begin{remark}
Rank–1 projections in $\textup{M}_n(\mathbb{R}),\textup{M}_n(\mathbb{C})$ are equivalent to projections whose images are $1$–dimensional. However, $\textup{M}_n(\mathbb{H})$ doesn't have any rank–1 projections. Nevertheless, in the context of this paper, the space of projections whose images are $1$–dimensional as right $\mathbb{H}$–modules behaves similarly enough.
\end{remark}
Any rank–1 projection defines a map of $Z(\mathcal{A})$–algebras. Quantum mechanically, they are non–Hermitian states:
\begin{definition}\label{rhoboat}
For $q\in\mathds{1}(\mathcal{A}),$ we define a $Z(\mathcal{A})$–module map
\begin{equation}
\rho_q:\mathcal{A}\to Z(\mathcal{A})\,,
\end{equation}
where $\rho_q(x)$ is the unique solution of 
\begin{equation}
qxq=\rho_q(x)q\,.
\end{equation}
\end{definition}
\subsubsection{Tangent Bundle of Functors}
Since projection space defines a functor, we can define its tangent bundle by analogy with the functor of points of the tangent bundle of a scheme (\ref{mum}):
\begin{definition}
For a commutative ring with involution $(\mathcal{R},*),$ consider a functor
\begin{equation}
\mathcal{F}:\textup{CAlg}_{(\mathcal{R},*)}\to   \textup{Set}_{*}\,.
\end{equation}
We define its tangent bundle to be the functor 
\begin{equation}
\textup{T}\mathcal{F}:\textup{CAlg}_{(\mathcal{R},*)}\to   \textup{Set}_{*}\,,\;\;\textup{T}\mathcal{F}(\mathcal{S},*)=\mathcal{F}(\mathcal{S}[\varepsilon]/(\varepsilon^2), *)\,,   
\end{equation}
where $(s+\varepsilon t)^*=s^*+\varepsilon t^*.$ The tangent space $\textup{T}_p\mathcal{F}$ at $p\in \mathcal{F}(\mathcal{S}, *)$ is the fiber of 
\begin{equation}
 \textup{T}\mathcal{F}(\mathcal{S},*)\to \mathcal{F}(\mathcal{S},*)   
\end{equation}
over $p.$
\end{definition}
\begin{definition}
Let $\mathcal{F},\mathcal{G}$ be functors $\textup{CAlg}_{(\mathcal{R},*)}\to   \textup{Set}_{*}.$
We say that a natural transformation 
\begin{equation}
\eta:\mathcal{F}\to\mathcal{G}
\end{equation}
is \'{e}tale if for any $(\mathcal{S},*)\in\textup{CAlg}_{(\mathcal{R},*)}$ and any $p\in \mathcal{F}(\mathcal{S},*),$ the induced map
\begin{equation}
\textup{T}_p\mathcal{F}(\mathcal{S},*)\to \textup{T}_{\eta_{(\mathcal{S},*)}(p)}\mathcal{G}(\mathcal{S},*)
\end{equation}
is a bijection.
\end{definition}
\subsubsection{Tangent Bundle of Projection Space}
We have the following simple characterization of the tangent bundle of projection space:
\begin{proposition}\label{vectoridentity}
Let $a\in\mathcal{A}.$ Then $a\in \textup{T}_q\mathbb{I}(\mathcal{A})$ if and only if $aq+qa=a.$ In particular, for all $x\in\mathcal{A},$ $[q,x]\in \textup{T}_q\mathbb{I}(\mathcal{A}).$ 
\end{proposition}
\begin{proof}
The first part follow from the expansion $(q+\varepsilon a)^2=q^2+\varepsilon (aq+qa).$ The second part follows from the first part by computing $q[q,x]+[q,x]q.$
\end{proof}
In the following sense, $\mathds{1}(\mathcal{A})$ is a connected component of $\mathbb{I}(\mathcal{A})$:
\begin{proposition}
If $q\in \mathds{1}(\mathcal{A})$ and $a\in\textup{T}_q\mathbb{I}(\mathcal{A})$ then $a\in\textup{T}_q\mathds{1}(\mathcal{A}).$
\end{proposition}
\begin{proof}
Let $q\in\mathds{1}(\mathcal{A}),\,a\in \textup{T}_q\mathbb{I}(\mathcal{A}),\,x\in\mathcal{A}.$
Then using the defining property of rank–1 projections and that $aq+qa=a,$ we compute 
\begin{align}
    (q+\varepsilon a)x(q+\varepsilon a)=&\rho_q(x)q+\varepsilon (axq+qxa)
    \\&=\rho_q(x)q+\varepsilon (\rho_q(ax)q+\rho_q(xa)q+\rho_q(x)aq+\rho_q(x)qa)
\\&=\Big(\rho_q(x)+\varepsilon \big(\rho_q(ax)+\rho_q(xa)\big)\Big)(q+\varepsilon a)\,.
\end{align}
Therefore, the map 
\begin{equation}
    Z(\mathcal{A}[\varepsilon]/(\varepsilon^2))\to(q+\varepsilon a)\mathcal{A}[\varepsilon]/(\varepsilon^2)(q+\varepsilon a)
    \end{equation}
is surjective. It is injective since $(r+\varepsilon s)(q+\varepsilon a)=0$ implies that $rq=sq=0,$ which implies that $r=s=0.$
\end{proof}
\subsection{Geometry of Rank–1 Projections}
For a rank–1 projection $q,$ we can rewrite \cref{themainl} as 
\begin{equation}
    \mathcal{L}_q(a)=2\rho_q(q^*)[a^*,q]+\rho_q(a^*)[q,q^*].
\end{equation}
\Cref{mainnn} then implies that $\textup{T}_q\mathds{1}(\mathcal{A})$ is naturally a representation of 
\begin{equation}
Z(\mathcal{A})[x,*]/(x^2+4\rho_q(q^*)^3)\,,
\end{equation}
where $Z(\mathcal{A})[x,*]$ is the skew–polynomial ring, ie. $xr=r^*x$ for $r\in Z(\mathcal{A}).$
There is a compatible closed 2–form:
\begin{definition}
We define a pointwise alternating and $Z(\mathcal{A})$–bilinear map 
\begin{equation}
\omega:\textup{T}\mathds{1}(\mathcal{A})\otimes \textup{T}\mathds{1}(\mathcal{A})\to Z(\mathcal{A})\,,\;\;\omega_q(a,b)=\rho_q([a,b])\,.
\end{equation}
\end{definition}
\begin{definition}
Let $(\mathcal{A},*)$ be a $^*$–ring. A trace is a $Z(\mathcal{A})$–module morphism
\begin{equation}
    \rho:\mathcal{A}\to Z(\mathcal{A})
\end{equation}
such that $\rho(xy)=\rho(yx)$ for all $x,y\in\mathcal{A}$ and such that $\rho(x^*)=\rho(x)^*$ for all $x\in\mathcal{A}.$  
\end{definition}
\begin{proposition}
Let $\rho$ be a trace on $\mathcal{A}.$ Then for all $q\in\mathds{1}(\mathcal{A})$ and $x\in\mathcal{A},$
\begin{equation}
    \rho(qx)=\rho_q(x)\rho(q)\,.
\end{equation}
Furthermore, for all $q\in\mathbb{I}(\mathcal{A})$ and $a\in\textup{T}_q\mathbb{I}(\mathcal{A}),$ $\rho(a)=0.$
\end{proposition}
\begin{proof}
The first part follows from applying $\rho$ to both sides of $qxq=\rho_q(x)q.$ The second part follows from the equation $qa+aq=a$ and the fact that $qaq=0.$
\end{proof}
As a result, $x\mapsto \rho_q(x),\,x\mapsto\rho(qx)$ differ by the multiplication of a ``locally constant" function. In particular, given a trace we can define a closed 2–form on $\mathbb{I}(\mathcal{A})$:
\begin{definition}
Let $\rho$ be a trace on $\mathcal{A}.$ We define 
\begin{equation}
\Omega:\textup{T}\mathbb{I}(\mathcal{A})\otimes \textup{T}\mathbb{I}(\mathcal{A})\to Z(\mathcal{A})\,,\;\;\Omega_q(a,b)=\rho(q[a,b])\,.
\end{equation}
\end{definition}For matrix algebras over commutative rings, $\omega=\Omega\vert_{\mathds{1}(\mathcal{A})}$ with respect to the usual trace. Because of this and the minor simplifications it provides, we will switch the emphasis from the map $x\to\rho_q(x)$ to the map $x\to\rho(qx).$
\\\\Note that, $\Omega$ naturally extends to a 2–form $\textup{T}\mathcal{A}\otimes\textup{T}\mathcal{A}\to Z(\mathcal{A})$ and we can associate to it a 3–form, its ``exterior derivative", ie. 
\begin{equation}
\textup{T}\mathcal{A}\otimes\textup{T}\mathcal{A}\otimes\textup{T}\mathcal{A}    \to Z(\mathcal{A})\,,\;\;(a,b,c)\mapsto \rho(a[b,c])-\rho(b[a,c])+\rho(c[a,b])  \,. 
\end{equation} 
Its pullback to $\mathbb{I}(\mathcal{A})$ is zero:
\begin{proposition}\label{closed}
For $a,b,c\in\textup{T}_q\mathbb{I}(\mathcal{A}),$
\begin{equation}\label{closed}
\rho(a[b,c])-\rho(b[a,c])+\rho(c[a,b])=0\,.
\end{equation}
\end{proposition}
\begin{proof}
We have $abc=(qa+aq)(qb+bq)(qc+cq)=qabqc+aqbcq.$ Therefore, by the cyclic property of $\rho,$
\begin{equation}
\rho(abc)=\rho(qabqc)+\rho(aqbcq)=0\,,
\end{equation}
which implies that each term in \cref{closed} is zero. 
\end{proof}
\begin{proposition}
Let $\rho$ be a trace on $\mathcal{A}$ such that $\rho\vert_{\mathds{1}(\mathcal{A})}=1.$ Then for all $a,b\in\textup{T}_q\mathds{1}(\mathcal{A}),$
\begin{equation}
  \Omega_q(\mathcal{L}_q(a),b)=\Omega_q(\mathcal{L}_q(b),a)^*\,.  
\end{equation}
\end{proposition}
For matrix $^*$–rings over quaternion algebras, this result is still true when replacing $\mathds{1}(\mathcal{A})$ with projections whose images are one–dimensional. Furthermore, the assumption $\rho\vert_{\mathds{1}(\mathcal{A})}=1$ is only used to slightly simplify formulas. 
\begin{proof}
Under the assumptions, $\rho(qm)=\rho_q(m)$ for all $m\in\mathcal{A}.$ Together with $qaq=qbq=0,$ we compute
\begin{align}
    \rho([\mathcal{L}_q(a),b])&=2\rho(q[[a^*,qq^*q],b])+\rho([q[qa^*q,q^*],b])
    \\&=-2\rho(q^*q)(\rho(qba^*)+\rho(qa^*b))+\rho((qa^*qq^*b)+\rho(qbq^*qa^*)
    \\&=-2\rho(q^*q)(\rho(a^*qb)+\rho(a^*bq))+\rho(qa^*)\rho(q^*bq)+\rho(qa^*)\rho(q^*qb)
    \\&=-2\rho(q^*q)\rho(a^*b)+\rho(qa^*)\rho(q^*b)\,.
\end{align}
On the other hand,
\begin{align}
 \rho([a,\mathcal{L}_q(b)])=2\rho(q^*q)\rho(b^*a)-\rho(qb^*)\rho(q^*a)\,.
\end{align}
Using the antisymmetry of $\Omega$ completes the proof.
\end{proof}
Due to the previous result, we can define an inner product on the tangent spaces that is compatible with $(*,\,\Omega,\,\mathcal{L})$:
\begin{definition}\label{metrich}
Let $\rho$ be a trace on $\mathcal{A}$ such that $\rho_{\mathds{1}(\mathcal{A})}=1.$ We define
\begin{equation}
    h:\textup{T}_q\mathds{1}(\mathcal{A})\otimes \textup{T}_q\mathds{1}(\mathcal{A})\to Z(\mathcal{A})\,,\;\;h_q(a,b)=2\rho(q^*q)\rho(ab^*)-\rho(q^*a)\rho(qb^*)\,.
\end{equation}
\end{definition}
By the proof of the previous result:
\begin{corollary}\label{corh}
\begin{equation}
h_q(a,b)=\Omega_q(a,\mathcal{L}_q(b))\,.
\end{equation}
\end{corollary}
\begin{proposition}
Suppose that $2\rho_q(q^*)$ is not  a zero divisor and that $(x,y)\mapsto \rho(xy)$ is non–degenerate, ie. if $x\in\mathcal{A}$ is such that $\rho(xy)=0$ for all $y\in\mathcal{A},$ then $x=0.$ Then $h_q$ is non–degenerate.
\end{proposition}
\begin{proof}
Let $x\in \mathcal{A}.$ Then $[q,x]\in\textup{T}_q\mathds{1} (\mathcal{A})$ and we compute
\begin{align}
h_q(a,\mathcal{L}_q([q,x]))&=-4\rho_q(q^*)^3\rho(q[a,[q,x])=4\rho_q(q^*)^3\rho(qaxq+qxaq)=4\rho_q(q^*)^3\rho(ax)\,,
\end{align}
and this implies the result.
\end{proof}
Assuming non–degeneracy of $\rho,$ which is true for matrix algebras, the only thing preventing the triple $(h,\Omega,\mathcal{L})$ from being hyperk\"{a}hler–like is a lack of integrability of $\mathcal{L},$ which is related to the fact that $\mathcal{L}^2$ is non–constant. However, we can normalize it on a cover, as we do in the next section.
\section{Hyperk\"{a}hler–Like Geometry of the Covering Space of Projections}
Finally, we will explain the hyperk\"{a}hler–like structure associated with a $^*$–ring. Under some mild conditions that are always satisfied for a matrix $^*$–ring over a commutative ring with involution, the rank–1 projections are naturally equipped with a triple $(\tilde{h},\Omega,I)$ such that $h$ defines a Hermitian inner product on the tangent spaces,\footnote{By which we mean, a non–degenerate bilinear form such that $\langle b,a\rangle=\langle a,b\rangle^*.$} $\Omega$ is a closed non–degenerate 2–form, and $I$ is an integrable almost complex structure\footnote{By which we mean, $I^2=-1.$} such that for all $r\in Z(\mathcal{A}),\,Ir=r^*I.$ They satisfy $\tilde{h}=\Omega\circ I.$
First, we define some covering spaces:
\begin{definition}
Let $Z(\mathcal{A})_{\textup{sa}}$ denote the self–adjoint elements of $Z(\mathcal{A}).$ We define 
 \begin{align}
     &\tilde{\mathds{1}}(\mathcal{A}):=\{(q,r)\in\mathds{1}(\mathcal{A})\times Z(\mathcal{A})_{\textup{sa}}:r^2\rho_q(q^*)=1\}\,.
 \end{align}
\end{definition}
We will identify this space with a space of the following form:
\begin{definition}
Let $(\mathcal{A},*)$ be a $^*$–ring and let $\dagger$ be an involution on $\mathcal{A}$ that commutes with $*$ and satisfies $(xy)^\dagger=y^\dagger x^\dagger.$ Let $Z(\mathcal{A})_{\textup{sa}}$ denote the elements of $Z(\mathcal{A})$ that are fixed by both $*,\dagger.$ We define
\begin{equation}
 \tilde{\mathbb{P}}(\mathcal{A})=\{(p,r)\in\mathbb{P}(\mathcal{A})\times Z(\mathcal{A})_{\textup{sa}}:r^2\rho_p(p^\dagger)=1\}\,.   
\end{equation}
\end{definition}
If $2\in\mathcal{A}^{\times}$ then we can define an almost complex structure on $\tilde{\mathds{1}}(\mathcal{A}),$ due to the following:
\begin{proposition}
Suppose $2\in \mathcal{A}^\times.$ Then the projection onto the first factor, $\tilde{\mathds{1}}\to\mathds{1},$ is \'{e}tale.
\end{proposition}
\begin{proof}
For any such $^*$–ring $(\mathcal{A},*)$ and $(q,r)\in\tilde{\mathds{1}}(\mathcal{A}),$
we have that
\begin{align}
&\textup{T}_{(q,r)}\tilde{\mathds{1}}(\mathcal{A})\cong\{(a,s)\in \textup{T}_q\mathds{1}(\mathcal{A})\times Z(\mathcal{A}):2r^{-1}s+r^2\big(\rho_q(a^*)+\rho_q(aq^*)+\rho_q(q^*a)\big)=0\}\,.
\end{align}
The result now follows from the fact that for any $a\in\textup{T}_q\mathds{1}(\mathcal{A})$ there is a unique $s$ such that $(a,s)$ is in the set on the right, ie. 
\begin{equation}
s=-\frac{r^3}{2}\big(\rho_q(a^*)+\rho_q(aq^*)+\rho_q(q^*a)\big)\,.
\end{equation}
\end{proof}
\begin{definition}
With respect to the identification $\textup{T}_{(q,r)}\tilde{\mathds{1}}(\mathcal{A})\cong \textup{T}_{q}\mathds{1}(\mathcal{A}),$ for all $(q,r)\in\tilde{\mathds{1}}(\mathcal{A})$ let     
\begin{equation}
I_{(q,r)}:\textup{T}_{(q,r)}\tilde{\mathds{1}}(\mathcal{A})\to \textup{T}_{(q,r)}\tilde{\mathds{1}}(\mathcal{A})\,,\;\;
I_{(q,r)}(a):=\frac{r^3}{2}\mathcal{L}_q(a)\,.
\end{equation}
\end{definition}
We can rewrite this as 
\begin{equation}
    I_{(q,r)}(a)=r[a^*,q]+\frac{r^3\rho_q(a^*)}{2}[q,q^*]\,.
\end{equation}
\begin{corollary}
For all $(q,r)\in\tilde{\mathds{1}}(\mathcal{A})$ and $a\in \textup{T}_{(q,r)}\tilde{\mathds{1}}(\mathcal{A})$ we have that $I^2_{(q,r)}(a)  =-a.$
\end{corollary}
Furthermore, we can normalize the $h$ of \cref{metrich} on this cover:
\begin{definition}\label{metricht}
Let $\rho$ be a trace on $\mathcal{A}$ such that $\rho_{\mathds{1}(\mathcal{A})}=1$ and assume that $2\in\mathcal{A}^\times.$ We define
\begin{equation}
    \tilde{h}:\textup{T}_{(q,r)}\tilde{\mathds{1}}(\mathcal{A})\otimes \textup{T}_{(q,r)}\tilde{\mathds{1}}(\mathcal{A})\to Z(\mathcal{A})\,,\;\;\tilde{h}_{(q,r)}(a,b)=r\rho(ab^*)-\frac{r^3}{2}\rho(q^*a)\rho(qb^*)\,.
\end{equation}
\end{definition}
It follows from \cref{corh} that
\begin{equation}
\tilde{h}_q(a,b)=\Omega_q(a,I_{(q,r)}(b))\,.
\end{equation}
Furthermore, $\tilde{\mathds{1}}(\mathcal{A})$ has an involution, given by $(q,r)^*=(q^*,r).$ Related to this is the following:
\begin{lemma}
Let $\mathcal{A}$ be a $^*$–ring and assume that $2\in\mathcal{A}^\times$ and let $x^*=-x.$ Then the map 
\begin{equation}\label{rootinv}
    \tilde{\mathds{1}}(\mathcal{A}[x]/(x^2+1))\to \tilde{\mathds{1}}(\mathcal{A}[x]/(x^2+1))\,,\;\;(w,r)\mapsto\Big(\frac{w+w^*}{2}+x\frac{r[w,w^*]}{2},r\Big)
\end{equation}
is well–defined and its square is given by $(w,r)\mapsto (w^*,r).$
\end{lemma}
We have the following complex embedding result, which shows that $(\tilde{\mathds{1}}(\mathcal{A}),I)$ is analogous to a complex affine manifold.
\begin{corollary}\label{bic}
Let $(\mathcal{A},*)$ be a $^*$–ring and assume that $2\in\mathcal{A}^\times.$ Consider the $^*$–ring $\mathcal{A}[x]/(x^2+1),$ with involution given by $(a+xb)^*=a^*+xb^*.$ Let $\dagger$ be the involution $(a+xb)^\dagger=a^*-xb^*.$ The map of \ref{rootinv} restricts to a bijection  
\begin{equation}
    \tilde{\mathds{1}}(\mathcal{A})\to\tilde{\mathbb{P}}(\mathcal{A}[x]/(x^2+1)])
\end{equation}
and its formal derivative intertwines $I$ and $x.$\footnote{The formal derivative is the map obtained by replacing $\mathcal{A}$ with $\mathcal{A}[\varepsilon]/(\varepsilon^2).$}
\end{corollary}
\begin{proof}
We need to show that the formal derivative intertwines $I$ and $x.$ Its differential is given by
\begin{align}
a\mapsto &\frac{a+a^*}{2}+\frac{x}{2}\Big(r[a,q^*]-r[a^*,q]-r^3(qa^*qq^++qq^*aq^*)\Big)
\\&=\frac{a+a^*}{2}-\frac{x}{2}(I(a)+I(a)^*)\,.
\end{align}
Applying $x$ to the right side, we get 
\begin{equation}
\frac{1}{2}(I(a)+I(a)^*)+\frac{x}{2}(a+a^*)
\end{equation}
which is the result obtained by applying the differential to $I(a).$
\end{proof}
Therefore, $\tilde{\mathds{1}}(\mathcal{A})$ comes with the analogues of a closed 2–form, an integrable almost complex structure and a K\"{a}hler potential (for the skew–adjoint part of $\Omega$), given by $(q,r)\mapsto r^{-1}.$ 
\begin{corollary}
Under the assumptions of the previous corollary, if $Z(\mathcal{A})$ contains an element $i$ such that $i^2=-1,$ then we get a well–defined map
\begin{equation}
  \tilde{\mathds{1}}(\mathcal{A})\to  \tilde{\mathds{1}}(\mathcal{A})\,,\;\;\Big(\frac{w+w^*}{2}+i\frac{r[w,w^*]}{2},r\Big)
\end{equation}
whose formal derivative intertwines $I$ and $i.$
\end{corollary}
As a result, with respect to both of the almost complex structures $i,I$ on $\textup{T}^*\mathbb{CP}^{n-1},$ $q\mapsto \sqrt{\textup{Tr}(q^\dagger q)}$ is a K\"{a}hler potential for the same Riemannian metric.
\\\\Finally, we end this section with the following, which shows that $^*$ is anti–symplectic.
\begin{lemma}\label{pol}
$\Omega_{q^*}(a^*,b^*)=-\Omega_q(a,b)^*\,.$
\end{lemma}
Related to this, $\mathds{1}(\mathcal{A}),\,\tilde{\mathds{1}}(\mathcal{A})$ are naturally para–K\"{a}hler manifolds, ie. they have a pair of transverse Lagrangian polarizations, which are related by $*.$ For the latter, these Lagrangian submanifolds are the fibers of the maps
\begin{equation}
    \tilde{\mathds{1}}(\mathcal{A})\to \mathbb{P}(\mathcal{A})\,,\;\;(q,r)\mapsto r^2qq^*\,,\;(q,r)\mapsto r^2q^*q\,.
\end{equation}
The corresponding splitting of the tangent bundle is given by $a=qa+aq.$ The corresponding involution of the tangent bundle is given by $a\mapsto[q,a].$
\section{The Scheme of Rank–1 Projections}
Considering example \ref{matrixex}, we give the following definition:
\begin{definition}
For a commutative ring $\mathcal{R},$ we let $\mathcal{R}\mathds{1}^{n-1}\subset M_n(\mathcal{R})$ denote the set of rank–1 projections. For a commutative ring with involution $(\mathcal{R},*),$ we let $\mathcal{R}_{*}\mathbb{P}^{n-1} \subset M_n(\mathcal{R})$ denote the set of self–adjoint rank–1 projections.
\end{definition}
We note that the image of an idempotent is automatically a projective module: $\mathcal{R}^n\cong \textup{im}(q)\oplus\textup{ker}(q).$ It follows that there is an embedding $\mathcal{R}_{*}\mathbb{P}^{n-1}\xhookrightarrow{}\mathcal{R}\mathbb{P}^{n-1},$ given by $q\mapsto \textup{im}(q).$ 
\begin{exmp}
When $\mathcal{R}$ is a field, the image of $\mathcal{R}_{*}\mathbb{P}^{n-1}\xhookrightarrow{}\mathcal{R}\mathbb{P}^{n-1}$ contains all rank–1 subspaces whose non–zero vectors satisfy $v^\dagger v\ne 0.$ In particular, letting $*$ be complex conjugation, $\mathbb{C}_*\mathbb{P}^n\cong \mathbb{CP}^n.$
\end{exmp}
\begin{exmp}
From \cref{bic}, we get a biholomorphism between $\textup{T}^*\mathbb{CP}^n\cong \mathbb{C}\mathbb{I}^{n}$ and bicomplex projective space, ie. $\mathbb{C}[x]/(x^2+1)_*\mathbb{P}^{n},$ where $(a+xb)^*=\bar{a}+x\bar{b}.$
\end{exmp}
\begin{proposition}
Using \cref{rhoboat}, $\rho_q(x)=\textup{Tr}(qx).$ In particular, $\Omega_q(a,b)=\textup{Tr}(q[a,b]).$
\end{proposition}
\begin{proof}
This follows from the fact that $\rho_q(x)q=qxq\implies \rho_q(x)\textup{Tr}(q)=\textup{Tr}(qx),$ together with the fact that $\textup{Tr}(q)=1,$ as we will see in \cref{repre}.
\end{proof}
\begin{exmp}
We can identify $\mathbb{R}\mathds{1}^1$ with $\textup{T}^*\mathbb{R}\mathbb{P}^1\cong S^1\times\mathbb{R}.$ Then  
\begin{equation}
    \Omega=\frac{1}{2}d\theta\wedge dt
\end{equation}
and the adjoint map is
\begin{equation}
    (e^{i\theta},t)^*=\Big(e^{i\theta}\frac{1+it}{1-it},-t\Big)\;.
\end{equation}
Equivalently, 
\begin{equation}
    (\theta,t)^*=(\theta+2\arctan{t},-t)\;.
\end{equation}
\end{exmp}
\begin{lemma}\label{1equiv}
Let $q\in\textup{M}_n(\mathcal{R})$ be an idempotent. Then $q\in\mathcal{R}\mathds{1}^{n-1}$ if and only if $\textup{im}(q)$ is rank–1 as a projective module.
\end{lemma}
\begin{proof}
This follows from the fact that a projective $\mathcal{R}$–module $M$ is rank–1 if and only if it is invertible, which is true if and only if $\textup{End}_{\mathcal{R}}(M)=\mathcal{R}.$  
\\\\To see how the result follows from this fact, we first note that $Z(\textup{M}_n(\mathcal{R}))\cong \mathcal{R}.$ Now, suppose $q\in\mathcal{R}\mathds{1}^{n-1}.$ Let $T\in \textup{End}_{\mathcal{R}}(\textup{im}(q))$ and extend it to $\widetilde{T}\in\textup{End}_{\mathcal{R}}(\mathcal{R}^n)$ by defining
\begin{equation}
  \widetilde{T}= Tq  
\end{equation}
For $v\in\textup{im}(q),$ we have that $\rho_q(\widetilde{T})v=q \widetilde{T}qv=Tv,$ hence $T=\rho_q(\widetilde{T})\in\mathcal{R}.$ Since $q\ne 0,$ this completes one direction.
\\\\Conversely, suppose that $\textup{End}_{\mathcal{R}}(\textup{im}(q))=\mathcal{R}$ and let $T\in\textup{End}_{\mathcal{R}}(\mathcal{R}^n).$ Then $qTq\in \textup{End}_{\mathcal{R}}(\textup{im}(q)),$ and therefore there exists some $r\in\mathcal{R}$ such that $qTq=rq.$ This completes the proof. 
\end{proof}
In the following, $\textup{CRing}$ is the category of commutative rings.
\begin{theorem}\label{repre}
The functor $\textup{CRing}\to\textup{Set},\,\mathcal{R}\to\mathcal{R}\mathds{1}^{n-1}$ is representable by
\begin{equation}
    \mathbb{Z}[\{x_j^i\}_{i,j=1}^n]/I\,,
\end{equation}
where $I$ is the ideal generated by 
\begin{equation}\label{term}
\sum_{i=1}^n x^j_ix^i_k-x^j_k\,,\;\;\sum_{i=1}^n x^i_i-1\,,\;\; x^j_kx^l_m-x^j_lx^k_m\,,\;\;\textup{for all}\;1\le j,k,l,m\le n\,.
\end{equation}
\end{theorem}
Setting each term of \cref{term} to zero, the first condition is the statement that a matrix is an idempotent, the second condition is the statement that its trace is $1$ and the third condition is the statement that all of its $2\!\times\!2$ minors vanish. The second and third statement together imply that it is rank–1.
\begin{proof}
We assume \cref{1equiv}. This result  then follows from localization and the fact that it's true when $\mathcal{R}$ is a local ring:
\\\\Suppose that $\mathcal{R}$ is a commutative ring and that $x\in \textup{M}_n(\mathcal{R})$ is an idempotent such that $\textup{im}(x)$ is rank–1 as a projective module. Let $\mathfrak{p}$ be a prime ideal. Then $x/1\in\mathcal{R}_{\mathfrak{p}}$ is an idempotent whose image is free of rank–1, since projective modules over local rings are free. Therefore, there exists a basis $v$ for $\textup{im}(x/1)$ and all of the column vectors of $x/1$ are in $\mathcal{R}_{\frak{p}}v,$ which implies that all $2\!\times\!2 $ minors vanish. To see that $\textup{Tr}(x)=1,$ choose bases for $\textup{ker}(x),\,\textup{im}(x).$ These determine a basis for $\mathcal{R}^n_{\frak{p}},$ and in this basis $x/1$ is diagonal with one entry equaling $1$ and all other entries equaling zero. Hence, $x/1$ has trace equal to $1$ and all of its $2\!\times\! 2$ minors vanish, and since $\frak{p}$ is an arbitrary prime ideal, this proves one direction.
\\\\For the other direction, let $\frak{p}$ be a prime ideal. The dimension of $\textup{im}(x)$ at $\frak{p}$ equals the dimension of $\textup{im}(x_\frak{p})$ over the residue field, $\kappa(\frak{p}).$ Since all $2\!\times\!2$ minors of $x_\frak{p}$ vanish, its rank is at most $1.$ Therefore, since $\textup{Tr}(x_\frak{p})\ne 0,$ its rank must be exactly one. Since $\frak{p}$ was arbitrary, it follows that $\textup{im}(x)$ is rank–1. This completes the proof.
\end{proof}
Due to the previous result, we make the following definition:
\begin{definition}
The (affine) scheme of rank–1 projections is 
\begin{equation}
\textup{Spec}(\mathbb{Z}[\{x_j^i\}_{i,j=1}^n]/I)\,.
\end{equation} 
\end{definition}
A choice of an involution on $\mathcal{R}$ determines an involution on the $\mathcal{R}$–points of this scheme, ie. rank–1 projections in $\textup{M}_n(\mathcal{R}).$ Therefore, we obtain a functor from commutative rings with involutions to sets with involutions. In this context, \cref{mainnn} says the following:
\begin{corollary}\label{mainmatrix}
Let $(\mathcal{R},*)$ be a commutative ring with an involution and for all $a \in \textup{T}_q\mathcal{R}\mathds{1}^{n-1},$ let 
    \begin{equation}
\mathcal{L}_q(a)=2\textup{Tr}(q^{\dagger} q)[a^{\dagger},q]+\textup{Tr}(a^\dagger q)[q,q^\dagger]\,.
\end{equation} 
Then $ \mathcal{L}_q(a)\in  \textup{T}_q\mathcal{R}\mathds{1}^{n-1}$ and 
\begin{equation}
\mathcal{L}_q^2(a)=-4\textup{Tr}(q^\dagger q)^3a\,.
\end{equation}  
\end{corollary}
\begin{proof}
This follows from \cref{mainnn} and the fact that for any $q\in\mathcal{R}\mathds{1}^{n-1}$ and $x\in M_n(\mathcal{R}),$ $qxq=\textup{Tr}(qx)q.$
\end{proof}
\begin{remark}
Consider $\mathbb{C}\mathds{1}^{n-1}.$ We have already seen the following: $I,i$ anticommute; are integrable; the real and imaginary parts of $\Omega$ are closed; $g(\cdot,\cdot):=\textup{Re}(\Omega)(\cdot,I(\cdot))$ is a Riemannian metric; $g(\cdot,iI\cdot)=\textup{Im}(\Omega).$
To see that $\mathbb{C}\mathds{1}^{n-1}$ is hyperk\"{a}hler, it is enough to check that $g(\cdot,i\cdot)$ is closed. To see the latter, it is enough to find a K\"{a}hler potential for $g,$ with respect to $i.$ Such a K\"{a}hler potential is given by $\sqrt{\textup{Tr}(q^\dagger q)}\,.$ In \cref{isomem}, we will also show that this metric is complete. Since $\mathbb{C}\mathds{1}^{n-1}$ is diffeomorphic to $\textup{T}^*\mathbb{CP}^{n-1},$ it follows by uniqueness (\ref{dan}, \ref{kal}) that $g$ is isometric to Calabi's hyperk\"{a}hler metric. 
\end{remark}
\section{Proof of the Main Identity}\label{mainproof}
Before carrying on with the proof of \cref{mainnn}, we have the following proposition:
\begin{proposition}\label{21}
If $q\in\mathbb{I}(\mathcal{A})$ is idempotent and $qq^*q=rq$ for some $r\in Z(\mathcal{A}),$ then $rqq^*=r^*qq^*$ and $r^2q=rr^*q.$
\end{proposition}
\begin{proof}
Computing $qq^*qq^*$ and using that it is self–adjoint gives
\begin{equation}
rqq^*=r^*qq^*\,.
\end{equation}
Multiplying both sides of this equation by $q$ on the right gives 
\begin{equation}
  r^2q=rr^*q\,.  
\end{equation}
\end{proof}
For matrix $^*$–rings, $r=\textup{Tr}(q^\dagger q)$ and it is therefore self–adjoint.  
\begin{proof}{of \cref{mainnn}:}
That $\mathcal{L}_q(a)\in\textup{T}_q\mathbb{I}(\mathcal{A})$ follows from \cref{vectoridentity}. That $\mathcal{L}_q(sa)=s^*\mathcal{L}_q(a)$ follows from the equation $(sa)^*=s^*a^*.$ For the final part, write
 \begin{equation}
\mathcal{L}_q(a)=2r[a^*,q]+[qa^*q,q^*]\,.
\end{equation}
We compute
 \begin{align}
    \mathcal{L}_q^2(a)&= 4r^2[[q^*,a],q] 
    \\&+\color{darkred}{2r[[q,q^*aq^*],q]}
    \\&+\color{darkblue}{2r^\#[q[q^*,a]q,q^*]}
    \\&\label{lastt}+\color{darkgreen}\cancelslash{red}{[q[q,q^*aq^*]q,q^*]}\,.
 \end{align}
\Cref{lastt} is zero since
$q\textup{T}_q\mathbb{I}(\mathcal{A}))q=0.$ Expanding the commutators, writing $a=qa+aq$ and using \cref{21}, we get
 \begin{align}
    \mathcal{L}_q^2(a)&\color{black}{=4r^2(\cancelslash{Magenta}{q^*aq}-\cancelslash{Slash5}{qaq^*q}-aqq^*q)}
    \\&\color{black}{-4r^2(qq^*qa+\cancelslash{orange}{qq^*aq}-\cancelslash{cyan}{qaq^*})}
    \\&\color{darkred}{+2r^2(\cancelslash{Slash5}{qaq^*q}-\cancelslash{cyan}{qaq^*)}}
    \\&\color{darkred}{+\cancelslash{orange}{2r^2qq^*aq}-\cancelslash{Slash9}{2rqq^*aqq^*}}
    \\&\color{darkred}{+2r^2(\cancelslash{orange}{qq^*aq}-\cancelslash{Magenta}{q^*aq})}
     \\&\color{darkred}{-\cancelslash{black}{2rq^*qaq^*q}+\cancelslash{Slash5}{2r^2qaq^*q}}
    \\&\color{darkblue}{+\cancelslash{Slash9}{2rqq^*aqq^*}-\cancelslash{cyan}{2r^2qaq^*}}
     \\&\color{darkblue}{-2r^2\cancelslash{Magenta}{q^*aq}+\cancelslash{black}{2rq^*qaq^*q}}
     \\&=-4r^3(aq+qa)
     \\&=-4r^3a\,.
 \end{align}
\end{proof}
\section{Isometric Embeddings}\label{isomem}
Let $\mathcal{V}$ be a finite–dimensional Hilbert space and let $\mathcal{L}(\mathcal{V})$ denote the $C^*$–algebra of linear operators on $\mathcal{V}.$ The Riemannian metric of $\textup{T}^*\mathbb{P}\mathcal{V}\cong \mathds{1}(\mathcal{L}(\mathcal{V}))$\footnote{This is the real part of the metric in \cref{metricht}.} in is pulled back from a Riemannian metric on $\mathcal{L}(\mathcal{V})-\{0\}.$ Therefore, in order to isometrically embed $\textup{T}^*\mathbb{P}\mathcal{V}$ into Euclidean space, it is enough to isometrically embed $\mathcal{L}(\mathcal{V})$ into Euclidean space. The aforementioned Riemannian metric on  $\mathcal{L}(\mathcal{V})-\{0\}$ is of the following form, with respect to the Hilbert–Schmidt inner product:
\begin{definition}
Let $\mathcal{H}$ be a real, complex or quaternionic Hilbert space. We define a Riemannian metric $g$ on $\mathcal{H}-\{0\}$ by 
\begin{equation*}
g_x(v,w):=\textup{Re}\bigg(\frac{2}{\|x\|}\langle v,w\rangle-\frac{1}{\|x\|^3}\langle v,x\rangle\langle x,w\rangle\bigg)\,.
\end{equation*}
\end{definition}
That $g$ is positive definite follows from the Cauchy–Schwartz inequality:
\begin{equation}
    \|v\|^2\|x\|^2\ge |\langle v,x\rangle|^2\implies \frac{\|v\|^2}{\|x\|}\ge \frac{|\langle v,x\rangle|^2}{\|x\|^3}
\end{equation}
and this implies that $g_x(v,v)>0$ for $v\ne 0.$ 
\begin{remark}
In the case of a complex Hilbert space, this metric is K\"{a}hler and the norm
\begin{equation}
\mathcal{H}-\{0\}\to\mathbb{R}\,,\;\; x\mapsto\|x\|    
\end{equation}
is a K\"{a}hler potential for $g.$
\end{remark}
Before stating the isometric embedding result, we make two comments:
\begin{enumerate}
    \item Given any finite dimensional Hilbert space $\mathcal{H},$ we let $g_\mathcal{H}$ denote the canonical translation invariant Riemannian metric on $\mathcal{H},$ ie. for $v,w\in \textup{T}_x\mathcal{H},$
    \begin{equation}
        g_{\mathcal{H}}(v,w)=\textup{Re}\langle v,w\rangle\,.
    \end{equation} 
    \item For $x\in\mathcal{H},$ let $x^*\in\mathcal{H}^*$ be defined by $x^*(y)=\langle x,y\rangle.$ Then for any $x,y\in\mathcal{H},\,xy^*\in\mathcal{L}(\mathcal{H})$ is defined by $z\mapsto \langle y,z\rangle x$ and the inner product is given by
    \begin{equation}
        \langle xy^*,wz^*\rangle_{\mathcal{L}(\mathcal{H})}=\langle x,w\rangle\langle z,y\rangle\,.
    \end{equation}
\end{enumerate}
\begin{lemma}
Let $\mathcal{H}$ be a real, complex or quaternionic Hilbert space\footnote{In the case of a real Hilbert space, another isometric embedding is given by $\mathcal{H}-\{0\}\to\mathbb{R}\oplus\mathcal{H},$ $x\mapsto \|x\|^{1/2}(1,\|x\|^{-1}x),$  where the metric on the codomain is $2g_{\mathbb{R}}\oplus 2g_{\mathcal{H}}.$} and let $\mathbb{R}\oplus\mathcal{H}\oplus\mathcal{L}(\mathcal{H})_{\textup{self–adjoint}}$ be equipped with the Riemannian metric given by 
\begin{equation}
    \frac{5}{2}g_{\mathbb{R}}\oplus g_{\mathcal{H}}\oplus\frac{1}{2}g_{\mathcal{L}(\mathcal{H})}\,.
\end{equation}
Then  
\begin{equation}
(\mathcal{H}-\{0\},g)\to \mathbb{R}\oplus\mathcal{H}\oplus\mathcal{L}(\mathcal{H})_{\textup{self–adjoint}}\,,\;\;x\mapsto \|x\|^{1/2}(1,\|x\|^{-1}x,\|x\|^{-2}xx^*)\,,
\end{equation}
is an isometric embedding.\footnote{This map is equivariant with respect to linear isometries of $\mathcal{H}$ and is positive homogeneous of degree $1/2.$}
\end{lemma}
\begin{proof}
The derivatives of $x\mapsto \|x\|^{1/2},\,x\mapsto \|x\|^{-1/2}x,\,x\mapsto \|x\|^{-3/2}xx^*$ are, respectively,
\begin{align}
    &v\mapsto \frac{1}{2}\|x\|^{-3/2}\textup{Re}\langle x,v\rangle\,,
    \\& v\mapsto \|x\|^{-1/2}v-\frac{1}{2}\|x\|^{-5/2}\textup{Re}\langle x,v\rangle x\,,
    \\&v\mapsto -\frac{3}{2}\|x\|^{-7/2}\textup{Re}\langle x,v\rangle xx^*+\|x\|^{-3/2}(vx^*+xv^*)\,.
\end{align}
It follows that the pullback metric is 
\begin{equation}
(v,w)\mapsto 2\|x\|^{-1}\textup{Re}\langle v,w\rangle-2\|x\|^{-3}\textup{Re}\langle x,v\rangle \textup{Re}\langle x,w\rangle  +\|x\|^{-3}\textup{Re}(\langle x,v\rangle\langle x,w\rangle)\,.
\end{equation}
Since $\textup{Re}(\overline{a}b)+\textup{Re}(ab)=2\textup{Re}(a)\textup{Re}(b),$ we can rewrite this as
\begin{equation}
(v,w)\mapsto 2\|x\|^{-1}\textup{Re}\langle v,w\rangle-\|x\|^{-3}\textup{Re}(\langle v,x\rangle\langle x,w\rangle)
\end{equation}
and this completes the proof.
\end{proof}
There is an action of $SU(n)\times\mathbb{Z}/2\mathbb{Z}$ on $\textup{T}^*\mathbb{CP}^{n-1}.$ Therefore, there is an action of $SU(n)\times\mathbb{Z}/2\mathbb{Z}\times \mathbb{Z}/2\mathbb{Z}$ on $\textup{T}^*\mathbb{CP}^{n-1}\times \{0,1\}.$ Due to the previous result:
\begin{corollary}
There is an $SU(n)\times\mathbb{Z}/2\mathbb{Z}\times \mathbb{Z}/2\mathbb{Z}$–equivariant isometric embedding
\begin{equation}\label{mainr}
    \textup{T}^*\mathbb{CP}^{n-1}\times\{0,1\}\xhookrightarrow{}\mathbb{R}^{(n^2+1)^2}
\end{equation}
whose image is an affine variety.
\end{corollary}
 Explicitly, the map \ref{mainr} is given by
\begin{align}
&\nonumber \mathds{1}(\mathcal{L}(\mathcal{H}))\times\{0,1\}\to \mathbb{R}\oplus \mathcal{L}(\mathcal{H})\oplus \mathcal{L}(\mathcal{L}(\mathcal{H}))_{\textup{self–adjoint}}\,,
\\&
(q,k)\mapsto (-1)^k\|q\|^{1/2}(1,\|q\|^{-1}q,\|q\|^{-2}q\langle q,\cdot\rangle)\,,
\end{align}
where $\|q\|=\sqrt{\textup{Tr}(q^*q)},$ $q\langle q,\cdot\rangle(T):=q\langle q,T\rangle$ and the image of \ref{mainr} is given by 
\begin{equation}
     \{(x,q,p):  p(q)=xq,\,p^2=xp,\,xq^2=q,\,x\textup{Tr}\,q=1,\,\textup{Tr}\, p=x,\textup{Tr}(q^*q)=x^2\}\,.
\end{equation}
\section*{Acknowledgments}
Thank you to Omar Kidwai for your superb accommodations. Special thanks to Yuanyuan for all of your support.  
\section{References}
\begin{enumerate}
    \item\label{bou}Charles Boubel. The algebra of parallel endomorphisms of a pseudo-Riemannian metric: semi-simple part. Mathematical Proceedings of the Cambridge Philosophical Society 159 (2015): 219--237. DOI: 10.1017/S0305004115000304.
    \item\label{cal}Eugenio Calabi. M\'etriques k\"ahl\'eriennes et fibr\'es holomorphes. Annales scientifiques de l'\'Ecole Normale Sup\'erieure 12 (1979): 269--294. DOI: 10.24033/asens.1367.
    \item\label{dan}
Andrew Dancer and Andrew Swann. Hyperk\"ahler metrics associated to compact Lie groups. Mathematical Proceedings of the Cambridge Philosophical Society 120 (1996): 61--69.
\item\label{dun}
Maciej Dunajski and Paul Tod. Conformal and isometric embeddings of gravitational instantons. Geometry, Lie Theory and Applications 16 (2022): 21--48. DOI: 10.1007/978-3-030-81296-6\_2.
    \item\label{egu}
Tohru Eguchi and Andrew J. Hanson. Asymptotically flat self-dual solutions to Euclidean gravity. Physics Letters B 74 (1978): 249--251. DOI: 10.1016/0370-2693(78)90566-X.
\item\label{fed}
Boris V. Fedosov. A simple geometrical construction of deformation quantization. Journal of Differential Geometry 40 (1994): 213--238. DOI: 10.4310/JDG/1214455536.
\item\label{han}
Andrew J. Hanson and Ji-Ping Sha. Isometric Embedding of the A1 Gravitational Instanton. Memorial Volume for Kerson Huang (2017): 95--111. DOI: 10.1142/9789813207431\_0013.
\item\label{kal}
D. Kaledin. Hyperk\"ahler metrics on total spaces of cotangent bundles. In D. Kaledin and M. Verbitsky, eds., Hyperk\"ahler manifolds. Mathematical Physics 12 (1999). International Press, Cambridge, MA.
\item\label{leu}
CW. Leung and CK. Ng. Analytic bundle structure on the idempotent manifold. Monatsh Math
196, 103–133 (2021). https://doi.org/10.1007/s00605-021-01562-4
\item\label{maz} 
F. Mazzoli, A. Seppi and A. Tamburelli. Para‑hyper‑Kähler Geometry of the Deformation Space of Maximal Globally Hyperbolic Anti‑de Sitter Three‑Manifolds. Memoirs of the American Mathematical Society, Vol. 306, No. 1546, (2025). Providence, RI: American Mathematical Society.
\item\label{mum}
David Mumford. Lectures on Curves on an Algebraic Surface. Annals of Mathematics Studies, No. 59. Princeton, NJ: Princeton University Press, 1966.
\item\label{mur}
Gerard J. Murphy. C$^*$-Algebras and Operator Theory. Academic Press, 1990.
\end{enumerate}
\end{document}